\documentclass[11pt,leqno]{article}

\usepackage{amsmath,amssymb,amsthm}

\usepackage[margin=3cm]{geometry} 

\usepackage{titlesec,hyperref}

\usepackage{color}

\usepackage{fancyhdr}
\titleformat{\subsection}{\it}{\thesubsection.\enspace}{1pt}{}

\newtheorem{theo}{Theorem}[section]

\newtheorem{defi}[theo]{Definition}

\newtheorem{prop}[theo]{Proposition}
\newtheorem{rema}[theo]{Remark}
\numberwithin{equation}{section}

\allowdisplaybreaks 

\begin{document}
\title{Analyticity of the Cauchy problem and persistence properties for a generalized Camassa-Holm equation
\hspace{-4mm}
}

\author{Xi $\mbox{Tu}^1$\footnote{E-mail: tuxi@mail2.sysu.edu.cn} \quad and\quad
 Zhaoyang $\mbox{Yin}^{1,2}$\footnote{E-mail: mcsyzy@mail.sysu.edu.cn}\\
 $^1\mbox{Department}$ of Mathematics,
Sun Yat-sen University,\\ Guangzhou, 510275, China\\
$^2\mbox{Faculty}$ of Information Technology,\\ Macau University of Science and Technology, Macau, China}

\date{}
\maketitle
\hrule

\begin{abstract}
This paper is mainly concerned with the Cauchy problem for a generalized Camassa-Holm equation with analytic initial data. The analyticity of its solutions is proved in both variables, globally in space and locally in time.
 Then, we present a persistence property for strong solutions to the system. Finally, explicit asymptotic profiles illustrate the optimality of these results.
 \\
\vspace*{5pt}
\vspace*{5pt}
\\\end{abstract}

\vspace*{10pt}

\tableofcontents

\section{Introduction}
In this paper we consider the Cauchy problem for the following generalized Camassa-Holm equation,
\begin{align}\label{E1}
\left\{
\begin{array}{ll}
u_t-u_{txx}=\partial_x(2+\partial_x)[(2-\partial_x)u]^2,~~~~  t>0,\\[1ex]
u(0,x)=u_{0}(x),
\end{array}
\right.
\end{align}
or equivalently
\begin{align}\label{E2}
\left\{
\begin{array}{ll}
m=u-u_{xx},\\[1ex]
m_t=2m^2+(8u_x-4u)m+(4u-2u_x)m_x+2(u+u_x)^2, ~~ t>0,\\[1ex]
m(0,x)=u(0,x)-u_{xx}(0,x)=m_0(x).
\end{array}
\right.
\end{align}

Note that $G(x)=\frac{1}{2}e^{-|x|}$ and $G(x)\star f =(1-\partial _x^2)^{-1}f$ for all $f \in L^2(\mathbb{R})$ and $G \star m=u$. Then we can rewrite (\ref{E1}) as follows:
\begin{align}\label{E03}
\left\{
\begin{array}{ll}
u_t(t,x)&=4uu_x+G\star[\partial_x(2u_x^2+6 u^{2})+\partial_x^2(u^2_{x})]
\\&=4uu_x-u^2_{x}+G\star[\partial_x(2u_x^2+6 u^{2})+u^2_{x}]
,~~~~  t>0,\\[1ex]
u(0,x)&=u_{0}(x).
\end{array}
\right.
\end{align}

The equation (\ref{E1}) was proposed recently by Novikov in \cite{n1}. It is integrable and belongs to the following class \cite{n1}:
\begin{align}\label{E02}
(1-\partial^2_x)u_t=F(u,u_x,u_{xx},u_{xxx}),
\end{align}
which has attracted much interest, particularly in the possible integrable members of (\ref{E02}).

The most celebrated integrable member of (\ref{E02}) is the well-known Camassa-Holm (CH) equation \cite{Camassa}:
\begin{align}
(1-\partial^2_x)u_t=3uu_x-2u_{x}u_{xx}-uu_{xxx}.
\end{align}

The CH equation can be regarded as a shallow water wave equation \cite{Camassa, Constantin.Lannes}.  It is completely integrable. It is completely integrable. That means that the system can be transformed into a linear flow at constant speed in suitable action-angle variables (in the sense of infinite-dimensional Hamiltonian systems), for a large class of initial data \cite{Camassa,Constantin-P,Constantin.mckean}. It also has a bi-Hamiltonian structure \cite{Constantin-E,Fokas}, and admits exact peaked solitons of the form $ce^{-|x-ct|}$ with $c>0$, which are orbitally stable \cite{Constantin.Strauss}. It is worth mentioning that the peaked solitons present the characteristic for the travelling water waves of greatest height and largest amplitude and arise as solutions to the free-boundary problem for incompressible Euler equations over a flat bed, cf. \cite{Camassa.Hyman,Constantin2,Constantin.Escher4,Constantin.Escher5,Toland}.

The local well-posedness for the Cauchy problem of the CH equation in Sobolev spaces and Besov spaces was discussed in \cite{Constantin.Escher,Constantin.Escher2,d1,Guillermo}. It was shown that there exist global strong solutions to the CH equation \cite{Constantin,Constantin.Escher,Constantin.Escher2} and finite time blow-up strong solutions to the CH equation \cite{Constantin,Constantin.Escher,Constantin.Escher2,Constantin.Escher3}. The existence and uniqueness of global weak solutions to the CH equation were proved in \cite{Constantin.Molinet, Xin.Z.P}. The global conservative and dissipative solutions of CH equation were investigated in \cite{Bressan.Constantin,Bressan.Constantin2}. Finite propagation speed and persistence properties of solutions to the Camassa-Holm equation have been studied in \cite{Constantin1,H-M-P}.

The second celebrated integrable member of (\ref{E02}) is the famous Degasperis-Procesi (DP) equation \cite{D-P}:
\begin{align}\label{dp}
(1-\partial^2_x)u_t=4uu_x-3u_{x}u_{xx}-uu_{xxx}.
\end{align}
The DP
equation can be regarded as a model for nonlinear shallow water
dynamics and its asymptotic accuracy is the same as for the
CH shallow water equation \cite{D-G-H}. The DP equation is integrable and has a bi-Hamiltonian structure \cite{D-H-H}. An inverse scattering approach for the
DP equation was presented in \cite{Constantin.lvanov.lenells,Lu-S}.
Its traveling wave solutions was investigated in \cite{Le,V-P}. \par The
local well-posedness of the Cauchy problem of the DP equation in Sobolev spaces and Besov spaces was established in
\cite{G-L,H-H,y1}. Similar to the CH equation, the
DP equation has also global strong solutions
\cite{L-Y1,y2,y4} and finite time blow-up solutions
\cite{E-L-Y1, E-L-Y,L-Y1,L-Y2,y1,y2,y3,y4}. It also has global weak
solutions \cite{C-K,E-L-Y1,y3,y4}.

\par
Although the DP equation is similar to the
CH equation in several aspects, these two equations are
truly different. One of the novel features of the DP
different from the CH equation is that it has not only
peakon solutions \cite{D-H-H} and periodic peakon solutions
\cite{y3}, but
also shock peakons \cite{Lu} and the periodic shock waves \cite{E-L-Y}.

The third celebrated integrable member of (\ref{E02}) is the known Novikov equation \cite{n1}:
\begin{align}
(1-\partial^2_x)u_t=3uu_{x}u_{xx}+u^2u_{xxx}-4u^2u_x.
\end{align}
The most difference between the Novikov equation and the CH and DP equations is that the former one has cubic nolinearity and the latter ones have quadratic nolinearity.

It was showed that the Novikov equation is integrable, possesses a bi-Hamiltonian structure, and admits exact peakon solutions $u(t,x)=\pm\sqrt{c}e^{|x-ct|}$ with $c>0$ \cite{Hone}.\\
$~~~~~~$ The local well-posedness for the Novikov equation in Sobolev spaces and Besov spaces was studied in \cite{Wu.Yin2,Wu.Yin3,Wei.Yan,Wei.Yan2}. The global existence of strong solutions were established in \cite{Wu.Yin2} under some sign conditions and the blow-up phenomena of the strong solutions were shown in \cite{Wei.Yan2}. The global weak solutions for the Novikov equation were studied in \cite{Laishaoyong,Wu.Yin}.

Recently, the Cauchy problem of (\ref{E1}) in the Besov spaces $B^{s}_{p,r},~s>max\{\frac{1}{p},\frac{1}{2}\}$ and critical Besov space $B^{\frac{1}{2}}_{2,1}$ has been studied in \cite{Tu-Yin1,Tu-Yin2}.
To our best knowledge, analyticity and persistence properties of the Cauchy problem for (\ref{E1}) has not been studied yet.
In this paper we first prove the analyticity of solutions to the system (\ref{E1}) in both variables,
with $x$ on the line $\mathbb{R}$ and $t$ in an neighborhood of zero, provided that the initial data is analytic on $\mathbb{R}$. Furthermore, we show a persistence property of the strong solutions to (\ref{E1}). The analysis of the solutions in weighted spaces is useful to obtain information on their spatial asymptotic behavior.

The paper is organized as follows. In Section 2 we obtain the analytic solutions to the equation (\ref{E1}) on the line. In Section 3 we obtain a persistence result on solutions to (\ref{E1}) in the weight $L^p$ spaces $L^p,\phi:=L^p(\mathbb{R},\phi^ pdx).$ In Section 4, we compute the spatial asymptotic profiles.

\vspace*{2em}
{\bf
Notations.}
Since all space of functions in the following sections are over $\mathbb{R},$ for simplicity, we drop $\mathbb{R}$ in our notations of function spaces if there is no ambiguity.

\section{Analytic solutions}
Setting
$$
f(x)=x^2,~~P_{1}=\partial_{x},~~P_{2}=\partial_{x}(1-\partial^2_{x})^{-1},
$$
we can rewrite (\ref{E03}) in the following form:
\begin{align}\label{31}
\partial_{t}u=2P_{1}(f(u))+2P_{2}(f(u_{x}))+6P_{2}(f(u))+P_{2}P_1(f(u_{x})).
\end{align}
Next, we can transform (\ref{31}) into the following equation:
\begin{align}
\nonumber\partial_{t}u_{1}&=2P_{1}(f(u_{1}))+2P_{2}(f(u_{2}))+6P_{2}(f(u_{1}))+P_{1}P_{2}(f(u_{2}))
\\\nonumber&=P_{1}\bigg(2f(u_{1})\bigg)+P_{2}\bigg(2f(u_{2})+6f(u_{1})\bigg)+P_{1}P_{2}\bigg(f(u_{2})\bigg)
\\\nonumber&=F_{1}(u_{1},u_{2}),
\\\nonumber\partial_{t}u_{2}&=4P_{1}(u_{1}u_{2})-P_{1}(f(u_{2}))+P_{1}P_{2}\bigg(2f(u_{2})+6f(u_{1})\bigg)+P_{2}(f(u_{2}))
\\\nonumber&=F_{2}(u_{1},u_{2}).
\end{align}

Therefore we can get the following system from combining the above equation with our initial value:

\begin{align}\label{E31}
\left\{
\begin{array}{ll}
\partial_{t}u_{1}=F_{1}(u_{1},u_{2}),\\[1ex]
\partial_{t}u_{2}=F_{2}(u_{1},u_{2}),\\[1ex]
u_{1}(0,x)=u_{0}(x),\\[1ex]
u_{1}(0,x)=\partial_{x}u_{0}(x).\\[1ex]
\end{array}
\right.
\end{align}

Define
$U = (u_1, u_2)$ and $F(U) = F(u_1, u_2)=(F_1(u_1, u_2), F_2(u_1, u_2))$.
Then we have
\begin{align}\label{E32}
\left\{
\begin{array}{ll}
\frac{\partial U(t)}{\partial_{t}}=F(t,U(t)),\\[1ex]
U(0)=(u_{0}(x),\partial_{x}u_{0}(x)).\\[1ex]
\end{array}
\right.
\end{align}

Before stating our result, we need to introduce some suitable Banach spaces. For any $s > 0$, we
define
$$E_{s}\triangleq \bigg{\{} u\in C^{\infty}:\|u\|_{s}=\sup_{k\in\mathbb{N}}  \frac{s^k\|\partial^{k}_{x}u\|_{H^1}}{k!/(k+1)^2} <\infty     \bigg{\}}.
$$

Now we have the following analyticity result:
\begin{theo}\label{th1}
Let $u_{0}\in E_{s}$. There exist an $\varepsilon>0$ and a unique solution $u$ of the Cauchy problem (\ref{E03}) that is analytic on $(-\varepsilon,~\varepsilon)\times \mathbb{R}$.
\end{theo}

We first study the properties of the space $E_s$. Let us recall some useful properties of
these Banach spaces, as follows.

\begin{prop}\label{p1}\cite{b-g1,h-m1,y-y1}
\\(1) $E_s$ equipped the norm
$||\cdot||_s$ is a Banach space by the completeness of $H^1 $ and the closedness of the differential
operator $\partial_{x}$.
\\(2) The functions in $E_s$ is real analytic on $\mathbb{R}$, namely, $E_{s} \subset C^{\omega}$.
\\(3) For any $0 < s' < s$, $E_s \hookrightarrow E_{s'} $ with $||u||_{s'} \leq ||u||_{s}$.
\\(4) Let $s > 0$. There exists a constant $C > 0$, independent of $s$, such that for any $u$ and $v$ in $E_s$, we get
$$||uv||_s \leq C||u||_s||v||_s,$$
where $C = C(r)$ depends only on $r$.
In particular, for any $s > 0$, we have
$$||f (u)- f (v)||_s = ||u^2 - v^2||_{s} \leq C||u + v||_{s}||u - v||_{s},$$
for any $u,~v \in E_s$.
\\(5) For any $0 < s' < s \leq 1$, we have
\begin{align}
||P_{1}u||_{s'} \leq \frac{1}{s-s'}||u||_s,\\
||P_2u||_s \leq  ||u||_s.
\end{align}
\end{prop}

Then Theorem \ref{th1} is a straightforward consequence of the following result.
\begin{theo}\cite{b-g1,h-m1}
 Let $(X_{s}, ||\cdot||_{s})_{0<s\leq1} $ be a scale of decreasing Banach spaces, such that $X_{s} \subset X_{s'}$ with $||\cdot||_{s'}\leq ||\cdot||_{s}$ for any $0 < s' < s$. Consider the Cauchy problem
\begin{align}\label{E33}
\left\{
\begin{array}{ll}
\frac{\mathrm{d} u}{\mathrm{d}_{t}}=F(t,u(t)),\\[1ex]
u(0)=0.\\[1ex]
\end{array}
\right.
\end{align}
Let $T,~ R ~and~ C$ be positive numbers and suppose that $F$ satisfies the following conditions:
\\(a) If for any $0 < s' < s < 1$, the function $t\mapsto u(t)$ is holomorphic on $|t| < T$ and continuous on $|t|\leq T$ with values in $X_{s}$
and $$\sup_{|t|\leq T}||u(t)||_{s} < R,$$
then $t \mapsto F (t, u(t))$ is a holomorphic function on $|t| < T$ with values in $X_{s'}$ .
\\(b) For any $0 < s' < s\leq 1$ and any $u,~v \in B(0, R) \subset X_{s}$, that is, $||u||_s < R,~ ||v||_s < R$, we have
$$\sup_{|t|\leq T}
||F (t, u)-F (t, v)||_{s'}\leq\frac{C}{s - s'}||u-v||_{s}.$$
\\(c) There exists a $M > 0$, such that for any $0 < s < 1$,
$$\sup_{|t|\leq T}||F (t, 0)||_{s} \leq \frac{M}{1 - s}.$$
\end{theo}

Now we set
$$||(u_1, u_2)||_{s}\triangleq \sum_{i=1}^{2}||u_{i}||_{s},$$
and
$$||F (u_1, u_2)||_{s} \triangleq \sum_{i=1}^{2}||F_{i}(u_1, u_2)||_{s}.$$

\begin{proof}
To complete the proof of Theorem \ref{th1}, it suffices to verify the conditions $(a)-(c)$ above.

We first focus on checking the conditions $(a)$. To check the condition $(a)$, we only need to derive that
 $\|F (U)\|_{s'}<\infty,$ when $\sup_{|t|\leq T}\|U\|_{s}<\infty.$

For any $u_j \in E_{s},~ (j = 1, 2),$ we obtain
\begin{align}
\nonumber\|F (U)\|_{s'}=& \|F (u_1, u_2)\|_{s'}
\\\nonumber=&\|F_1(u_1, u_2)\|_{s'}+\|F_2(u_1, u_2)\|_{s'}
\\\nonumber=&\|P_{1}\bigg(2f(u_{1})\bigg)+P_{2}\bigg(2f(u_{2})+6f(u_{1})\bigg)+P_{1}P_{2}\bigg(f(u_{2})\bigg)\|_{s'}
\\\nonumber&+\|4P_{1}(u_{1}u_{2})-P_{1}(f(u_{2}))+P_{1}P_{2}\bigg(2f(u_{2})+6f(u_{1})\bigg)+P_{2}(f(u_{2}))\|_{s'}
\\\nonumber=&\frac{2}{s-s'}\|u_{1}^{2}\|_{s}+2\|u_{2}^{2}\|_{s'}+6\|u_{1}^{2}\|_{s'}+\frac{1}{s-s'}\|u_{2}^{2}\|_{s}
\\\nonumber&+\frac{4}{s-s'}\|u_{1}u_{2}\|_{s}+\frac{1}{s-s'}\|u_{2}^{2}\|_{s}+\frac{2}{s-s'}\|u_{2}^{2}\|_{s}
+\frac{6}{s-s'}\|u_{1}^{2}\|_{s}+\|u_{2}^{2}\|_{s'}
\\\nonumber\leq&\frac{C}{s-s'}(\|u_{1}\|_{s}+\|u_{2}\|_{s})^2
\\\nonumber\leq&\frac{C}{s-s'}\|(u_{1},u_{2})\|^2_{s},
\end{align}
which means that system (\ref{E32}) with zero initial datum satisfies the condition $(a)$.
This completes the proof of the condition $(a)$.

Next, we consider the condition $(b)$.

For any $u_j$ and $v_j \in B(0, R) \subset E_{s},~ (j = 1, 2),$ we get
\begin{align}
\nonumber\|F (U)- F (V)\|_{s'}=& \|F (u_1, u_2) - F(v_1, v_2)\|_{s'}
\\\nonumber&=\|F_1(u_1, u_2) - F_1(v_1, v_2)\|_{s'}+\|F_2(u_1, u_2) - F_2(v_1, v_2)\|_{s'}
\\\nonumber&\triangleq I_1 + I_2.
\end{align}

Then, we will estimate $I_1$ and $I_2$ respectively. By Proposition \ref{p1}, yields
\begin{align}
I_{1}\nonumber\leq& \|2P_{1}[f(u_{1})-f(v_{1})]\|_{s'}+\|2P_{2}[f(u_{2})-f(v_{2})]\|_{s'}
\\\nonumber&+6\|P_{2}[f(u_{1})-f(v_{1})]\|_{s'}+\|P_{1}P_{2}[f(u_{2})-f(v_{2})]\|_{s'}
\\\nonumber\leq&\frac{2}{s-s'}\|f(u_{1})-f(v_{1})\|_{s}+2\|f(u_{2})-f(v_{2})\|_{s'}
\\\nonumber&+6\|f(u_{1})-f(v_{1})\|_{s'}+\frac{1}{s-s'}\|f(u_{2})-f(v_{2})\|_{s}
\\\nonumber\leq&\frac{2C}{s-s'}\|u_{1}+v_{1}\|_{s}\|u_{1}-v_{1}\|_{s}+2C\|u_{2}+v_{2}\|_{s'}\|u_{2}-v_{2}\|_{s'}
\\\nonumber&+6C\|u_{1}+v_{1}\|_{s'}\|u_{1}-v_{1}\|_{s'}+\frac{C}{s-s'}\|u_{2}+v_{2}\|_{s}\|u_{2}-v_{2}\|_{s}
\\\leq&\frac{C(r,R)}{s-s'}||(u_1, u_2)-(v_1, v_2)||_{s},
\\\nonumber I_{2}\leq& \|4P_{1}(u_{1}u_{2}-v_{1}v_{2})\|_{s'}+\|P_{1}[f(u_{2})-f(v_{2})]\|_{s'}
\\\nonumber&+\|2P_{1}P_{2}[f(u_{2})-f(v_{2})]\|_{s'}
\\\nonumber&+\|6P_{1}P_{2}[f(u_{1})-f(v_{2})]\|_{s'}+\|P_{2}[f(u_{2})-f(v_{2})]\|_{s'}
\\\nonumber\leq&\frac{4}{s-s'}\|u_{1}u_{2}-v_{1}v_{2}\|_{s}+\frac{1}{s-s'}\|f(u_{2})-f(v_{2})\|_{s}
\\\nonumber&+\frac{2}{s-s'}\|f(u_{2})-f(v_{2})\|_{s}
\\\nonumber&+\frac{6}{s-s'}\|f(u_{1})-f(v_{1})\|_{s}+\|f(u_{2})-f(v_{2})\|_{s'}
\\\nonumber\leq&\frac{4}{s-s'}(\|u_{2}\|_{s}\|u_{1}-v_{1}\|_{s}+\|v_{1}\|_{s}\|u_{2}-v_{2}\|_{s})
+\frac{C}{s-s'}\|u_{2}+v_{2}\|_{s}\|u_{2}-v_{2}\|_{s}
\\\nonumber&+\frac{2C}{s-s'}\|u_{2}+v_{2}\|_{s}\|u_{2}-v_{2}\|_{s}
\\\nonumber&+\frac{6C}{s-s'}\|u_{1}+v_{1}\|_{s}\|u_{1}-v_{1}\|_{s}+C\|u_{2}+v_{2}\|_{s'}\|u_{2}-v_{2}\|_{s'}
\\\leq&\frac{C(r,R)}{s-s'}||(u_1, u_2)-(v_1, v_2)||_{s}.
\end{align}

This completes the proof of the condition $(b)$.

The condition $(c)$ can be easily obtained once  our system (\ref{E31}) or (\ref{E32}) is transformed into a new system with zero initial data as in (\ref{E33}). This complete the proof of Theorem \ref{th1}.

\end{proof}

\section{Persistence properties}
In this section, we shall discuss the persistence properties for a generalized Camassa-Holm equation (\ref{E1}) in weighted $L^p$ spaces.
We can first draw some standard definitions.
\begin{defi}
(1) In general a weight function is simply a non-negative function.
\\(2) A weight function $v : \mathbb{R} \rightarrow \mathbb{R}$ is sub-multiplicative if
$$v(x + y)\leq v(x)v(y),~~ \forall x, y \in \mathbb{R}.$$
(3) Given a sub-multiplicative function $v$, by definition a positive function $\phi$ is $v$-moderate
if and only if
$$\exists~~ C_{0} > 0:~ \phi(x + y) \leq C_{0} ~v(x)\phi(y), ~~~\forall x, y \in \mathbb{R}.$$
We say that $\phi$ is moderate which means $\phi$ is $v$-moderate for some sub-multiplicative function $v$.
\end{defi}
Let us recall the most standard examples of such weights:
$$\phi(x) = \phi_{a,b,c,d}(x) = e^{a|x|^b}(1 + |x|)^c \log(e + |x|)^d.
$$
We have (see \cite{Brandolese}):
\\(1) For $a, c, d \geq 0$ and $ 0\leq b \leq 1$, such weight is sub-multiplicative.
\\(2) If $a, c, d \in \mathbb{R}$ and $0 \leq b \leq 1$, then $\phi$ is moderate. More precisely, $\phi_{a,b,c,d} $ is $\phi_{\alpha,~\beta,~\gamma,~\delta} -moderate$ for
$|a| \leq \alpha,~ b \leq \beta,~ |c| \leq \gamma,~and~ |d|\leq \delta $.

Then, we give the definition for an admissible weight function:
\begin{defi}\label{def1}
An admissible weight function for (\ref{E1}) is a locally absolutely continuous function $\phi:\mathbb{R}\rightarrow\mathbb{R}$ which satisfy
\\(1) $\phi$ is $v$-moderate where $v$ is some sub-multiplicative weight function such that
\begin{align}
\nonumber
\inf_{x\in\mathbb{R}}v(x)>0,
\end{align}~~and~~
\begin{align}\label{052}
\int_{\mathbb{R}}\frac{v(x)}{e^{|x|}}dx<\infty.
\end{align}
\\(2) $|\phi'(x)|\leq A|\phi(x)|$, for some $A>0$ and a.e. $x\in \mathbb{R}$.
\end{defi}
Let us recall the following useful proposition which will be used in the proof of persistence properties.
\begin{prop}\label{066}\cite{Brandolese}
\\(1) Let $v : \mathbb{R}^{n} \rightarrow \mathbb{R}^{+} $ and $C_0 > 0$. Then following conditions are equivalent:

 (i) $\forall~ x, y : v(x + y) \leq C_{0} v(x)v(y)$.

 (ii) For all $1 \leq p, q, r \leq \infty $ and for any measurable functions $f_{1}, f_{2} : \mathbb{R}^n \rightarrow \mathbb{C} $ the weighted Young inequalities hold:
$$\|(f_1 \ast f_2)v\|_{r} \leq C_{0}\|f_1v\|_{p}\|f_2v\|_{q},~~~~~ 1 +\frac{1}{r}=\frac{1}{p}+\frac{1}{q}.$$
\\(2) Let $1 \leq p \leq \infty$ and $v$ be a sub-multiplicative weight on $\mathbb{R}^{n}$. The following two conditions are equivalent:

(i) $\phi$ is a $v-moderate$ weight function (with constant $C_0$).

(ii) For all measurable functions $f_1$ and $f_2$ the weighted Young estimate holds
 $$\|(f_1 \ast f_2)\phi\|_{p} \leq C_{0}\|f_1v\|_{1}\|f_2\phi\|_{p}.$$
\end{prop}

We can now state our main result on an admissible weight function.

\begin{theo} \label{thm5}
Let $T>0$, $s>\frac{5}{2}$, and $2\leq p \leq \infty$. Let also $u\in C([0,T],H^s)$ be a strong solution of the Cauchy problem for (\ref{E1}), such that $u|_{t=0}=u_{0}$ satisfies
$$u_{0}\phi \in L^{p},~~ ~(\partial_{x}u_{0})\phi \in L^{p}~~ ~and~ ~~(\partial_{xx}u_{0})\phi \in L^{p},
$$
where $\phi$ is an admissible weight function for (\ref{E1}). Then, for all $t\in ~[0,T]$, we have the estimate,
$$\|u(t)\phi\|_{p}+\|\partial_{x}u(t)\phi\|_{p}+\|\partial_{xx}u(t)\phi\|_{p}
\leq\bigg(\|u_{0}\phi\|_{p}+\|\partial_{x}u_{0}\phi\|_{p}+\|\partial_{xx}u_{0}\phi\|_{p}\bigg)e^{CMt},$$
for some constant  $C>0$ depending only on $v,~\phi(A,~~C_0,~~\inf_{x\in \mathbb{R}}v(x),~~and~~\\\int_{\mathbb{R}}\frac{v(x)}{e^{|x|}}dx<\infty)$, and
$$M\equiv \sup_{t\in[0,T]}(\|u(t)\|_{\infty}+\|\partial_{x}u(t)\|_{\infty}+\|\partial_{xx}u(t)\|_{\infty})<\infty.$$
\end{theo}

The standard weights $\phi=\phi_{a,b,c,d}(x)=e^{a|x|^b}(1+|x|)^clog(e+|x|)^d$ is the basic example of the application of Theorem \ref{thm5}, if it satisfies the following conditions:
$$a\geq0,~~~~~c,d\in \mathbb{R},~~~~~b=1,~~~~~ab<1.$$
The restriction $ab<1$ guarantees the validity of the condition (\ref{052}) for a multiplicative function $v (x)\geq 0$.
The restriction $b=1$ guarantees that $|\phi'(x)|\leq A|\phi(x)|.$ When $0< b<1$, we have $|\phi'|\rightarrow \infty$, as $x\rightarrow 0.$
\begin{proof}
Assume that $\phi$ is the admissible weight function.
Applying the assumption $u\in C([0,T],H^{s})$, with $s>\frac{5}{2}$, we get
$$M\equiv \sup_{t\in[0,T]}(\|u(t)\|_{\infty}+\|\partial_{x}u(t)\|_{\infty}+\|\partial_{xx}u(t)\|_{\infty})<\infty.$$

The first equation in (\ref{E03}) can be rewritten as:
\begin{align}\label{053}
u_t-4uu_{x} - \partial_{x}G\ast [2u_{x}^2+6u^2+\partial_{x}(u_{x}^2)]=0,
\end{align}
with the kernel $G(x)=\frac{1}{2}e^{-|x|}$.
Let us consider the N-truncations of $\phi_{N}(x)=\min\{\phi,N\}$ for any $N\in \mathbb{Z}^{+}.$
Note that $\phi_{N}:\mathbb{R}\rightarrow \mathbb{R}$ is a locally absolutely continuous function which satisfies
\begin{align}\label{056}
\|\phi_{N}(x)\|_{\infty}\leq N,~~~|\phi_{N}'(x)|\leq A|\phi_{N}(x)|~~~~a.e. on ~\mathbb{R},
\end{align}
and
$$\phi_{N}(x+y)\leq C_{1}v(x)\phi_{N}(y),~~~~\forall x,y\in \mathbb{R},$$
where $C_1=\max\{C_0,\alpha^{-1}\},~~\alpha=\inf_{x\in \mathbb{R}}v(x)>0.$

As known in \cite{Brandolese}, the $N$-truncations $\phi_{N}$ of a $v$-moderate weight $\phi$ are uniformly $v$-moderate with respect to $N$.

Considering the case $2\leq p<\infty$. First, we give estimates on $u\phi_{N}$.  Multiply (\ref{053}) by $\phi_{N}|u\phi_{N}|^{p-2}(u\phi_{N})$ and integrate to obtain
\begin{align}
\frac{1}{p}&\nonumber\frac{\mathrm{d}}{\mathrm{d}t}\bigg(\|u\phi_{N}\|_{p}^{p}\bigg)-
4\int_{\mathbb{R}}|u\phi_{N}|^{p}(\partial_{x}u)\mathrm{d}x
\\&-\int_{\mathbb{R}}|u\phi_{N}|^{p-2}(u\phi_{N})(\phi_{N} G_{x}\ast [2u_{x}^2+6u^2+\partial_{x}(u_{x}^2)])\mathrm{d}x=0.
\end{align}
The two above integrals are clearly finite since $\phi_{N}\in L^{\infty}(\mathbb{R})$. And $u(\cdot, t) \in H^s, s > \frac{5}{2}$ leads to
$\partial_{x}G\ast [2u_{x}^2+6u^2+\partial_{x}(u_{x}^2)] \in L^1\cap L^{\infty}.$
Applying H\"{o}lder's inequality, Proposition \ref{066} and the condition (\ref{052}), it follows that
\begin{align}\label{057}
\frac{\mathrm{d}}{\mathrm{d}t}\|u\phi_{N}\|_{p}\nonumber &\leq 4M\|u\phi_{N}\|_{p}+\|\phi_{N} G_{x}\ast [2u_{x}^2+6u^2+\partial_{x}(u_{x}^2)]\|_{p}
\\\nonumber &\leq 4M\|u\phi_{N}\|_{p}+C_{1}\|G_{x}v\|_{L^1}\|[2u_{x}^2+6u^2+\partial_{x}(u_{x}^2)]\phi_{N}\|_{p}
\\\nonumber &\leq 4M\|u\phi_{N}\|_{p}+C_{2}\| [2u_{x}^2+6u^2+\partial_{x}(u_{x}^2)]\phi_{N}\|_{p}
\\ &\leq M(4+6C_{2})\|u\phi_{N}\|_{p}+6MC_{2}\| u_{x}\phi_{N}\|_{p},
\end{align}
where $C_{2}$ depends only on $v$ and $\phi$.

Next, we will give estimates on $u_x\phi_{N}$. Differentiating (\ref{053}) with respect to $x$-variable, then multiplying by $\phi_{N}$, yields
\begin{align}\label{055}
\partial_{t}(u_{x}\phi_{N})-4u_{x}(u_{x}\phi_{N})-4u\phi_{N}\partial_{x}^2u +2u_{x}\phi_{N}\partial_{x}^2u- \phi_{N}\partial_{x}G\ast [\partial_{x}(2u_{x}^2+6u^2)+u_{x}^2]=0.
\end{align}
Multiplying the above equation by $|u_{x}\phi_{N}|^{p-2}(u_{x}\phi_{N})$, integrating the result equation in $x$-variable, and using integration by parts, we obtain
\begin{align}
\frac{1}{p}&\nonumber\frac{\mathrm{d}}{\mathrm{d}t}\bigg(\|u_{x}\phi_{N}\|_{p}^{p}\bigg)
-4\int_{\mathbb{R}}|u_{x}\phi_{N}|^{p}(\partial_{x}u)\mathrm{d}x
-4\int_{\mathbb{R}}\partial^2_{x}u|u_{x}\phi_{N}|^{p-2}(u_{x}\phi_{N})(u\phi_{N})\mathrm{d}x
\\&+2\int_{\mathbb{R}}|u_{x}\phi_{N}|^{p}(\partial^2_{x}u)\mathrm{d}x-\int_{\mathbb{R}}|u_{x}\phi_{N}|^{p-2}(u_{x}\phi_{N})
(\phi_{N} \partial_{x}G \ast [\partial_{x}(2u_{x}^2+6u^2)+u_{x}^2])\mathrm{d}x=0.
\end{align}
Arguing as before, the two above integrals are clearly finite because of $\phi_{N}\in L^{\infty}(\mathbb{R})$. And $u(\cdot, t) \in H^s, s > \frac{5}{2}$ leads to $\partial_{x}G\ast [\partial_{x}(2u_{x}^2+6u^2)+u_{x}^2] \in L^1\cap L^{\infty}.$
Applying H\"{o}lder's inequality, Proposition \ref{066} and condition (\ref{052}), we obtain
\begin{align}\label{058}
\frac{\mathrm{d}}{\mathrm{d}t}\|u_{x}\phi_{N}\|_{p}\nonumber &\leq 6M(\|u_{x}\phi_{N}\|_{p}+\|u\phi_{N}\|_{p})+\|\phi_{N} G_{x}\ast [\partial_{x}(2u_{x}^2+6u^2)+u_{x}^2]\|_{p}
\\\nonumber &\leq 6M(\|u_{x}\phi_{N}\|_{p}+\|u\phi_{N}\|_{p})+C_{1}\|G_{x}v\|_{L^1}\|
[\partial_{x}(2u_{x}^2+6u^2)+u_{x}^2]\phi_{N}\|_{p}
\\\nonumber &\leq 6M(\|u_{x}\phi_{N}\|_{p}+\|u\phi_{N}\|_{p})+C_{2}\| [\partial_{x}(2u_{x}^2+6u^2)+u_{x}^2]\phi_{N}\|_{p}
\\&\leq 6M\|u\phi_{N}\|_{p}+M(12C_{2}+6)\| u_{x}\phi_{N}\|_{p},
\end{align}
where $C_{2}$ depends only on $v$ and $\phi$.

Then, we will give estimates on $u_{xx}\phi_{N}$. Differentiating (\ref{053}) twice with respect to $x$-variable, next multiplying by $\phi_{N}$, we get
\begin{align}\label{44}
\nonumber\frac{\mathrm d}{\mathrm{d}t}(\phi_{N}u_{xx})(t)=&\phi_{N}(4u-2u_{x})u_{xxx}
-2\phi_{N}u^2_{xx}+8\phi_{N}u_{x}u_{xx}-12\phi_{N}uu_{x}
\\&+\phi_{N}\partial_{x}G\ast[(2u_x^2+6u^{2})+\partial_{x}(u^2_{x})].
\end{align}

Multiplying the above equation by $|\phi_{N}u_{xx}|^{p-2}(\phi_{N}u_{xx})$, integrating the result equations in the $x$-variable, and using integration by parts and $(\ref{056})$, implies
\begin{align}\label{059}
\frac{\mathrm{d}}{\mathrm{d}t}\|u_{xx}\phi_{N}\|_{p}\nonumber &\leq (4A+10)M\|\phi_{N}u_{xx}\|_{p}
+12M\|\phi_{N}u_{x}\|_{p}
+\|\phi_{N} G_{x}\ast [(2u_{x}^2+6u^2)+\partial_{x}(u_{x}^2)]\|_{p}
\\\nonumber &\leq (4A+10)M\|\phi_{N}u_{xx}\|_{p}+12M\|\phi_{N}u_{x}\|_{p}
+C_{1}\|G_{x}v\|_{L^1}\|
[\partial_{x}(2u_{x}^2+6u^2)+u_{x}^2]\phi_{N}\|_{p}
\\\nonumber &\leq (4A+10)M\|\phi_{N}u_{xx}\|_{p}+12M\|\phi_{N}u_{x}\|_{p}
+C_{2}\| [\partial_{x}(2u_{x}^2+6u^2)+u_{x}^2]\phi_{N}\|_{p}
\\ &\leq (4A+12)M\|\phi_{N}u_{xx}\|_{p}+24MC_{2}\| u_{x}\phi_{N}\|_{p},
\end{align}
where we use
\begin{align}\label{43}
\nonumber|\int_{\mathbb{R}}&\phi_{N}(2u_{x}-4u)u_{xxx}|\phi_{N}u_{xx}|^{p-2}(\phi_{N}u_{xx})\mathrm dx|
\\\nonumber
=&|\int_{\mathbb{R}}(2u_{x}-4u)[\partial_{x}(\phi_{N}u_{xx})-\phi_{N}'u_{xx}]|\phi_{N}u_{xx}|^{p-2}
(\phi_{N}u_{xx})\mathrm dx|
\\ \nonumber
\leq&|\frac{1}{p}\int_{\mathbb{R}}(2u_{x}-4u)\partial_{x}(|\phi_{N}u_{xx}|^{p})\mathrm dx|+|\int_{\mathbb{R}}(2u_{x}-4u)\phi_{N}'u_{xx}|\phi_{N}u_{xx}|^{p-2}(\phi_{N}u_{xx})\mathrm dx|
\\ \nonumber
\leq&|\frac{1}{p}\int_{\mathbb{R}}(2u_{xx}-4u_{x})(|\phi_{N}u_{xx}|^{p})\mathrm dx|+A\int_{\mathbb{R}}|2u_{x}-4u||\phi_{N}u_{xx}|^{p}\mathrm dx
\\ \nonumber
\leq& \frac{1}{p}(2\|u_{xx}\|_{L^\infty}+4\|u_{x}\|_{L^\infty})\|\phi_{N}u_{xx}\|_{L^{p}}^{p}
\\ \nonumber
&+A(2\|u_{x}\|_{L^\infty}+4\|u_{}\|_{L^\infty})\|\phi_{N}u_{xx}\|_{p}^{p}
\\ \leq& (4A+2)M\|\phi_{N}u_{xx}\|_{p}^{p}.
\end{align}

Now, together the inequalities (\ref{057}), (\ref{058}) with (\ref{059}) and then integrating, yields
\begin{align}
\|u&(t)\phi_{N}\|_{p}+\|(\partial_{x}u)\phi_{N}\|_{p}+\|(\partial_{xx}u)\phi_{N}\|_{p}\nonumber\\\nonumber&\leq (\|u_{0}\phi_{N}\|_{p}+\|(u_{0x})\phi_{N}\|_{p}+\|(u_{0xx})\phi_{N}\|_{p})\exp	(CMt)
\\&\leq (\|u_{0}\phi\|_{p}+\|(u_{0x})\phi\|_{p}+\|(u_{0xx})\phi\|_{p})\exp	(CMt),
\end{align}
for all $t\in[0,T]$.

Since $\phi_{N}(x) \rightarrow \phi(x)$ as $N\rightarrow \infty $ for a.e. $x\in\mathbb{R}$. Recalling that $u_{0}\phi \in L^{p}$, $u_{0x}\phi \in L^{p}$ and $u_{0xx}\phi \in L^{p}$, we get
\begin{align}
\|u&(t)\phi\|_{p}+\|(\partial_{x}u)\phi\|_{p}+\|(\partial_{xx}u)\phi\|_{p}\nonumber\\&\leq (\|u_{0}\phi\|_{p}+\|(u_{0x})\phi\|_{p}+\|(u_{0xx})\phi\|_{p})\exp(CMt),
\end{align}
for all $t\in[0,T]$.
Finally, we treat the case $p=\infty$. Noticing that $u_0,~u_{0x},~u_{0xx}\in L^2\cap L^\infty$ and $\phi_{N}\in L^\infty$, hence, we have
\begin{align}
\|u&(t)\phi_{N}\|_{q}+\|(\partial_{x}u)\phi_{N}\|_{q}+\|(\partial_{xx}u)\phi_{N}\|_{q}\nonumber\\&\leq (\|u_{0}\phi_{N}\|_{q}+\|(u_{0x})\phi_{N}\|_{q}+\|(u_{0xx})\phi_{N}\|_{q})\exp	(CMt)
\\\nonumber&\leq (\|u_{0}\phi_{N}\|_{\infty}+\|(u_{0x})\phi_{N}\|_{\infty}+\|(u_{0xx})\phi_{N}\|_{\infty})\exp(CMt),~~~q\in[2,\infty).
\end{align}

The last term in the right-hand side is independent of q. Since $\|\phi_{N}\|_{L^p}\rightarrow \|\phi_{N}\|_{L^\infty}$ as $p \rightarrow \infty$ for any $\phi_{N}\in L^\infty \cap L^2$, implies
\begin{align}
\|u&(t)\phi_{N}\|_{\infty}+\|(\partial_{x}u)\phi_{N}\|_{\infty}+\|(\partial_{xx}u)\phi_{N}\|_{\infty}\nonumber\\\nonumber&\leq (\|u_{0}\phi_{N}\|_{\infty}+\|(u_{0x})\phi_{N}\|_{\infty}+\|(u_{0xx})\phi_{N}\|_{\infty})\exp	(CMt)
\\&\leq (\|u_{0}\phi\|_{\infty}+\|(u_{0x})\phi\|_{\infty}+\|(u_{0xx})\phi\|_{\infty})\exp	 (CMt).
\end{align}

The last term in the right-hand side is independent of N. Now taking $N\rightarrow\infty$ implies that the estimate (2.6) remains valid for $p=\infty$.
\end{proof}

Indeed, for $a<0$, we have $\phi(x)\rightarrow 0$ as $|x|\rightarrow \infty$: the conclusion of Theorem \ref{thm5} remains true but we are not interested in this case. It is interesting in the following two particular cases:
\begin{rema}
(1) Power weights: Take $\phi=\phi_{0,0,c,0}$ with $c>0$. It is $v$-moderate, where $v=(1+|x|)^{c},$ $c<0.$ It can be deduced easily $v>0,~ve^{-|x|}\in L^{1}$. And sending $p=\infty$.

From Theorem \ref{thm5}, we obtain
$$|u(t,x)|+|\partial_{x}u(t,x)|+|\partial^2_{x}u(t,x)|\leq C'(1+|x|)^{-c}.
$$
with the condition
$$|u_0(x)|+|\partial_{x}u_{0}(x)|+|\partial^2_{x}u_{0}(x)|\leq C(1+|x|)^{-c}.$$

Thus, we obtain the persistence properties on algebraic decay rates of strong solutions to the (\ref{E1}).
\\(2) Exponential weights:
Choose $\phi=\phi_{a,1,0,0}.$ We deduce that $\phi(x)=e^{ax}$ if $x\geq0$ and $\phi(x)=1 $ if $x\leq0$ with $0\leq a<1$.
Such weight clearly satisfies the admissibility conditions of Definition \ref{def1}.

From Theorem \ref{thm5}, we have
$$|u(t,x)|+|\partial_{x}u(t,x)|+|\partial^2_{x}u(t,x)|\leq C'e^{-ax}.
$$
with the condition
$$|u_0(x)|+|\partial_{x}u_{0}(x)|+|\partial^2_{x}u_{0}(x)|\leq Ce^{-ax}.$$

Thus, we obtain the persistence properties on the point wise decay decay rates of strong solutions to the (\ref{E1}).
\end{rema}

The limit case $a = b = 1$ is not covered by Theorem \ref{thm5}.
In other words, To obtain an application of some $v$-moderate weights $\phi$ for which condition (\ref{052}) does not hold, we need to make a modification to Theorem \ref{thm5}. Instead of assuming (\ref{052}), we now put the weaker condition
\begin{align}\label{054}
ve^{-|\cdot|}\in L^p~~~~~where~~2\leq p\leq \infty.
\end{align}
See Theorem (\ref{thm6}) below, which covers the case of such fast growing weights.

\begin{theo}\label{thm6}
Let $2\leq p\leq \infty$ and $\phi$ be a $v$-moderate weight function as in Definition \ref{def1} satisfying the condition (\ref{054}) instead of (\ref{052}). Let also $u|_{t=0}=u_{0} $ satisfy
$$u_{0}\phi \in L^{p},~~~~u_{0}\phi^{\frac{1}{2}} \in L^{2},$$

$$(\partial_{x}u_{0})\phi \in L^{p},~~~~(\partial_{x}u_{0})\phi^{\frac{1}{2}} \in L^{2},$$
and
$$(\partial_{x}^2u_{0})\phi \in L^{p},~~~~(\partial^2_{x}u_{0})\phi^{\frac{1}{2}} \in L^{2}.$$

Let also $u\in C([0,T],H^s),s>\frac{5}{2}$, be the strong solution of the Cauchy problem for (\ref{E1}), emanating from $u_0$. Then,
$$\sup_{t\in[0,T]}	(\|u(t)\phi\|_{p}+\|\partial_{x}u(t)\phi\|_{p}+\|\partial^2_{x}u(t)\phi\|_{p})<\infty,$$
and
$$\sup_{t\in[0,T]}	(\|u(t)\phi^{\frac{1}{2}}\|_{2}+\|\partial_{x}u(t)\phi^{\frac{1}{2}}\|_{2}
+\|\partial^2_{x}u(t)\phi^{\frac{1}{2}}\|_{2})<\infty.$$
\end{theo}

\begin{proof}
It is easy to observe that $\phi ^{\frac{1}{2}}$ is a $v^{\frac{1}{2}}$-moderate weight where $\inf_{x\in \mathbb{R}} v^{\frac{1}{2}}(x) > 0$, such that
$|(\phi ^{\frac{1}{2}})'(x)| \leq \frac{A}{2}\phi ^{\frac{1}{2}}$.
By condition (\ref{054}), $v^{\frac{1}{2}}e^{-\frac{|x|}{2}} \in L^{2p}$,
hence H\"{o}lder's inequality implies that $v^{\frac{1}{2}}e^{-|x|} \in L^1$. Then an application
of Theorem \ref{thm5} with $p = 2$ to the weight $\phi^{\frac{1}{2}}$, yields
\begin{align}
\|u(t)\phi^{\frac{1}{2}}\|_{2}+\|\partial_{x}u(t)\phi^{\frac{1}{2}}\|_{2}+\|\partial^2_{x}u(t)\phi^{\frac{1}{2}}\|_{2}
\leq (\|u_{0}\phi^{\frac{1}{2}}\|_{2}+\|u_{0x}\phi^{\frac{1}{2}}\|_{2}+\|u_{0xx}\phi^{\frac{1}{2}}\|_{2})e^{CMt},
\end{align}
which along with H\"{o}lder's inequality, implies
\begin{align}\label{060}
\|[(2u_x^2+6u^{2})+\partial_{x}(u^2_{x})]\phi\|_{1}\leq K_{0}e^{2CMt},
\end{align}
\begin{align}\label{061}
\|[\partial_{x}(2u_x^2+6u^{2})+u^2_{x}]\phi\|_{1}\leq K_{1}e^{2CMt}.
\end{align}
The constants $K_0$ and $K_1$ below depend only on $\phi$ and on the datum.
By the same argument as in the proof of Theorem \ref{thm5} (recall
that $\phi_{N} = \min\{\phi(x),N\}$) for $p < \infty$,  yields
\begin{align}\label{064}
&\frac{\mathrm{d}}{\mathrm{d}t}\|u\phi_{N}\|_{p} \leq CM\|\phi_{N}u\|_{p}+\|\phi_{N} \partial_{x}G\ast [(2u_{x}^2+6u^2)+\partial_{x}(u_{x}^2)]\|_{p},
\\&\frac{\mathrm{d}}{\mathrm{d}t}\|u_{x}\phi_{N}\|_{p} \leq CM\|\phi_{N}u_{x}\|_{L^{p}}+\|\phi_{N} \partial_{x}G\ast [\partial_{x}(2u_{x}^2+6u^2)+(u_{x}^2)]\|_{p},
\end{align}
and \begin{align}\label{063}
\frac{\mathrm{d}}{\mathrm{d}t}\|u_{xx}\phi_{N}\|_{p} \leq CM\|\phi_{N}u_{xx}\|_{p}+\|\phi_{N}\partial_{x} G_{}\ast [(2u_{x}^2+6u^2)+\partial_{x}(u_{x}^2)]\|_{p}.
\end{align}

Note that $|\partial_{x}G| \leq \frac{1}{2}e^{-|\cdot|}$. Then combining Proposition \ref{066}, the condition (\ref{054}) with the estimates (\ref{060})-(\ref{061}), we have
$$\|\phi_{N} \partial_{x}G_{}\ast [\partial_{x}(2u_{x}^2+6u^2)+(u_{x}^2)]\|_{p}\leq K_{2}e^{2CMt},$$
$$\|\phi_{N}\partial_{x} G_{}\ast [(2u_{x}^2+6u^2)+\partial_{x}(u_{x}^2)]\|_{p}\leq K_{3}e^{2CMt}.$$
The constants $K_{2},~K_3$ depend only on $\phi$ and on the datum.

Plugging the two last estimates in (\ref{064})-(\ref{063}), and summing we obtain
\begin{align}
\frac{\mathrm{d}}{\mathrm{d}t}(&\|u\phi_{N}\|_{p}+\|u_{x}\phi_{N}\|_{p}+\|u_{xx}\phi_{N}\|_{p})\nonumber\\\nonumber\leq &CM(\|u\phi_{N}\|_{p}+\|u_{x}\phi_{N}\|_{p}+\|u_{xx}\phi_{N}\|_{p})+2(K_{2}+K_{3})e^{2CMt}.
\end{align}
Integrating and finally as before letting $N \rightarrow \infty$, we can get the conclusion in the case $ 2 \leq p < \infty$.
The constants throughout the proof are independent on $p$. Therefore, as before, the result for infinite exponents $p = \infty$ can be established. The argument is fully similar to that of Theorem \ref{thm5}.

\end{proof}

\begin{rema}
Choosing $\phi(x)=\phi_{1,1,0,0}(x)=e^{|x|}.$ It is $v$-moderate with $v=e^{|x|}$. Let $p=\infty.$ it is easy to deduce
that $\frac{v}{e^{|x|}}\in L^{\infty}$.
 Applying Theorem \ref{thm6}, we get the strong solution such that
$$|u(x,t)|+|\partial_{x}u(x,t)|+|\partial^2_{x}u(x,t)|\leq Ce^{-|x|},$$
if $|u_{0}(x)|$, $|\partial_{x}u_{0}(x)|$ and $|\partial^2_{x}u_{0}(x)|$ are both bounded by $ce^{-|x|}$.
\end{rema}

\section{Exact asymptotic profiles}
In this section, we finish this paper with the proof of a asymptotic profile.
\begin{theo}\label{thm10}
Let $s > \frac{5}{2}$ and $u_{0} \in H^s$, $u_{0} \neq 0$, such that
\begin{align} \label{l61}
\sup_{x\in\mathbb{R}}
e^{\frac{|x|}{2}}(1 + |x|)^{\frac{1}{2}} \log(e + |x|)^d (|u_{0}(x)| + |(\partial_{x}u_0)(x)|+|(\partial^2_{x}u_0)(x)|)< \infty,
\end{align}
for some $d >\frac{1}{2}$. Then condition (\ref{l61}) is conserved uniformly in $[0, T]$ by the strong solution $u \in C([0, T],H^s)$ of the Camassa-Holm equation. Moreover, the following asymptotic
profiles (respectively for $x \rightarrow +\infty$ and $x \rightarrow -\infty$) hold:
\begin{align}\label{l68}
\left\{
\begin{array}{ll}
u(x, t) = &u_0(x) + e^{-x}t\bigg\{\Phi(t) + C_{1}(1+x)^{-1}\log(e+x)^{-2d}(x, t)\\[1ex]&+C_{2}[\log(1+x)]^{1-2d}+o([\log(1+x)]^{1-2d})\bigg\},\\[1ex]
u(x, t) = &u_0(x) - e^{x}t \bigg\{\Psi(t) +  C_{1}(1-x)^{-1}\log(e-x)^{-2d}(x, t)\\[1ex]&+C_{2}[\log(1-x)]^{1-2d}+o([\log(1-x)]^{1-2d})\bigg\}.
\end{array}
\right.
\end{align}
where, for all $t \in [0, T]$, some constants $C_1, ~C_2$ depend only on the datum  and some constants $c_1, ~c_2 > 0$ independent on $t$,
\begin{align}\label{l63}
 c_1 \leq \Phi(t) \leq c_2,~~~ c_1 \leq \Psi(t) \leq c_2.
\end{align}
\end{theo}

\begin{proof}
A simple application of Theorem \ref{thm5} with $p = \infty$ and the weight
$\phi(x) = e^{\frac{|x|}{2}}(1 + |x|)^{\frac{1}{2}} \log(e + |x|)^d $ leads to the fact that condition (\ref{l61}) is conserved uniformly in [0, T] by the strong solution $u \in C([0, T],H^s).$

Integrating (\ref{E03}) over $[0,t]$, we obtain
\begin{align} \label{l62}
u(x, t) = \nonumber&u_0(x)+\int^{t}_{0}[4u(x, s) \partial_{x}u(x, s)-(\partial_{x}u(x, s))^2]ds
\\&-\int^{t}_{0}G \star ( \partial_{x}u)^{2}(x, s) ds+\int^{t}_{0}G_{x} \star F(u)(x, s) ds,
\end{align}
where $F(u) = 6u^2 + 2 (\partial_{x}u)^2$.

For $t \in [0, T]$, applying (\ref{l61}), we get
\begin{align}
\nonumber&|\int^{t}_{0}[4u(x, s) \partial_{x}u(x, s)-(\partial_{x}u(x, s))^2]ds|\leq \frac{C_{1}}{2}e^{-|x|}t(1+|x|)^{-1}\log(e+|x|)^{-2d},
\\\nonumber&|\int^{t}_{0}G \star ( \partial_{x}u)^{2}(x, s) ds|\leq \frac{C_{1}}{2}e^{-|x|}t(1+|x|)^{-1}\log(e+|x|)^{-2d}.
\end{align}

For $0 < t \leq T$, assume that
$$h(x, t) =\frac{1}{t}\int^{t}_{0}F(u)(x, s) ds,$$
and
$$\Phi(t) =
\frac{1}{2}
\int_{-\infty}^{+\infty}
e^yh(y, t) dy,~~~~~ \Psi(t) =
\frac{1}{2}
\int_{-\infty}^{+\infty}
e^yh(y, t) dy.$$
The function $(1 + | \cdot |)^{-\frac{1}{2}} \log(e + |\cdot|)^{-d}$ belongs to $L^2$. Then we have $e^{\frac{|x|}{2}}(|u_{0}(x)| + |(\partial_{x}u_0)(x)|+|(\partial^2_{x}u_0)(x)|)\in L^2.$ Applying Theorem \ref{thm5} with $p = 2$ and the weight $\phi(x) = e^{\frac{|x|}{2}}$ yields
$\int_{\mathbb{R}}e^{|y|}h(y, t) dy < \infty$.

Because $\Phi$ and $\Psi$ is continuous at $t = 0$, we give the definition
 $$\Phi(0) = \frac{1}{2}
\int_{-\infty}^{+\infty}e^{y}F(u_{0})(y) dy~~
and ~~\Psi(0) = \frac{1}{2}
\int_{-\infty}^{+\infty}e^{-y}F(u_{0})(y) dy.$$

The assumption $u_0\neq 0$ and $u_{0}\in H^s,~~s>\frac{5}{2}$,
ensures the validity of estimates (\ref{l63}) with $c_1, c_2 > 0$ for all $t\in [0, T]$.

Now using $\partial_{x}G(x- y) = -\frac{1}{2} sign(x - y)e^{-|x-y|}$ we get
\begin{align}
\nonumber-\int_{0}^{t}
\partial_{x}G \star F(u)(s) ds
&=\frac{e^{-x}t}{2}
\int_{-\infty}^{x}
e^{y}h(y, t) dy -
\frac{e^{x}t}{2}
\int_{x}^{+\infty}
e^{-y}h(y, t) dy
\\\nonumber&= e^{-x}t
[\Phi(t)-
\frac{1}{2}\int_{x}^{+\infty}
(e^{y} + e^{2x-y})h(y, t) dy].
\end{align}

But, by the dominated convergence theorem, it follows
\begin{align}
0 \leq
\int_{x}^{+\infty}
(e^y + e^{(2x-y)})h(y, t) dy \leq 2\int_{x}^{+\infty}e^y h(y, t) dy\rightarrow 0,~~~ as ~~x \rightarrow +\infty.
\end{align}
In the same way,
$$-\int_{0}^{t}
\partial_{x}G \star F(u)(s) ds = -e^{x}t[\Psi(t) -\frac{1}{2}\int_{-\infty}^{x}(e^{-y} + e^{y-2x})h(y, t) dy],
$$
and
$$
0\leq
\int_{x}^{-\infty}(e^{-y} + e^{y-2x})h(y, t) dy \leq 2
\int_{x}^{-\infty}2e^{-y}h(y, t) dy\rightarrow 0,~~~ as~ ~x \rightarrow-\infty .$$
By Hospital's Rule, we have
\begin{align}
\nonumber\lim_{x\rightarrow\infty}&\frac{\int_{x}^{+\infty}e^{|y|}h(t,y)dy}{[\log(1+|x|)]^{1-2d}}\\
=&\frac{-e^{|x|}h(t,x)}{(1-2d)[\log(1+|x|)]^{-2d}(1+|x|)^{-1}sign(x)}
\nonumber\\
\leq&C_{2}.
\end{align}

Thus, we complete the asymptotic profile of (\ref{l68}).

\end{proof}

Applying Theorem  \ref{thm10} we immediately obtain the fact that only the zero solution can be compactly supported (or decay faster than $e^{-|x|}$) at two different times $t_0$ and $t_1$.

\begin{theo}\label{thm11}
Let $s > \frac{5}{2}$ and $u_{0} \in H^s$, $u_{0} \neq \lambda e^{-\sqrt{3}x}$, such that
\begin{align} \label{l65}
\sup_{x\in\mathbb{R}}
e^{\frac{|x|}{2}}(1 + |x|)^{\frac{1}{2}} \log(e + |x|)^d (|u_{0}(x)| + |(\partial_{x}u_0)(x)|+|(\partial^2_{x}u_0)(x)|)< \infty,
\end{align}
for some $d >\frac{1}{2}$. Then condition (\ref{l65}) is conserved uniformly in $[0, T]$ by the strong solution $u \in C([0, T],H^s)$ of the Camassa-Holm equation. Moreover, the following asymptotic
profiles (respectively for $x \rightarrow +\infty$ and $x \rightarrow -\infty$) hold:
\begin{align}\label{l67}
\left\{
\begin{array}{ll}
u(x, t) = &u_0(x) + e^{-x}t\bigg\{\Phi(t) + C_{1}(1+x)^{-1}\log(e+x)^{-2d}(x, t)\\[1ex]&+C_{2}[\log(1+x)]^{1-2d}+o([\log(1+x)]^{1-2d})\bigg\},\\[1ex]
u(x, t) = &u_0(x) - e^{x}t \bigg\{\Psi(t) +  C_{1}(1-x)^{-1}\log(e-x)^{-2d}(x, t)\\[1ex]&+C_{2}[\log(1-x)]^{1-2d}+o([\log(1-x)]^{1-2d})\bigg\}.
\end{array}
\right.
\end{align}
where, for all $t \in [0, T]$, some constants $C_1, ~C_2$ depend only on the datum  and some constants $c_1, ~c_2 > 0$ independent on $t$,
\begin{align}\label{l64}
 c_1 \leq \Phi(t) \leq c_2,~~~ c_1 \leq \Psi(t) \leq c_2.
\end{align}
\end{theo}

\begin{proof}
A simple application of Theorem \ref{thm5} with $p = \infty$ and the weight
$\phi(x) = e^{\frac{|x|}{2}}(1 + |x|)^{\frac{1}{2}} \log(e + |x|)^d $ leads to the fact that condition (\ref{l65}) is conserved uniformly in [0, T] by the strong solution $u \in C([0, T],H^s).$

We can rewritten (\ref{E03}) as follows:
\begin{align}\label{e03}
u_t(t,x)=4uu_x-u^2_{x}+\sqrt{3}u^2-G\star[u^2_{x}-\sqrt{3}u^2]+G_{x}\star[(\sqrt{2}u_x+\sqrt{6} u)^2].
\end{align}

Integrating (\ref{e03}) over $[0,t]$, we obtain
\begin{align}
u(x, t) = \nonumber&u_0(x)+\int^{t}_{0}[4u(x, s) \partial_{x}u(x, s)-(\partial_{x}u(x, s))^2+\sqrt{3}u^2]ds
\\&-\int^{t}_{0}G \star [( \partial_{x}u)^{2}-\sqrt{3}u^2](x, s) ds+\int^{t}_{0}G_{x} \star[(\sqrt{2}u_x+\sqrt{6} u)^2](x, s) ds.
\end{align}

For $t \in [0, T]$, applying (\ref{l65}), we get
\begin{align}
\nonumber&|\int^{t}_{0}[4u(x, s) \partial_{x}u(x, s)-(\partial_{x}u(x, s))^2+\sqrt{3}u^2]ds|\leq \frac{C_{1}}{2}e^{-|x|}t(1+|x|)^{-1}\log(e+|x|)^{-2d},
\\\nonumber&|\int^{t}_{0}G \star [( \partial_{x}u)^{2}-\sqrt{3}u^2](x, s) ds|\leq \frac{C_{1}}{2}e^{-|x|}t(1+|x|)^{-1}\log(e+|x|)^{-2d}.
\end{align}

For $0 < t \leq T$, assume that
$$h(x, t) =\frac{1}{t^{\frac{1}{2}}}[\int^{t}_{0}(\sqrt{2}u_x+\sqrt{6} u)^2(x, s) ds]^{\frac{1}{2}},$$
and
$$\Phi(t) =
\frac{1}{2}
\int_{-\infty}^{+\infty}
e^yh(y, t) dy,~~~~~ \Psi(t) =
\frac{1}{2}
\int_{-\infty}^{+\infty}
e^yh(y, t) dy.$$
The function $(1 + | \cdot |)^{-\frac{1}{2}} \log(e + |\cdot|)^{-d}$ belongs to $L^2$. Then we have $e^{\frac{|x|}{2}}(|u_{0}(x)| + |(\partial_{x}u_0)(x)|+|(\partial^2_{x}u_0)(x)|)\in L^2.$ Applying Theorem \ref{thm5} with $p = 2$ and the weight $\phi(x) = e^{\frac{|x|}{2}}$ yields
$\int_{\mathbb{R}}e^{|y|}h(y, t) dy < \infty$.

Because $\Phi$ and $\Psi$ is continuous at $t = 0$, we give the definition
 $$\Phi(0) = \frac{1}{2}
[\int_{-\infty}^{+\infty}e^{y}(\sqrt{2}u_{0x}+\sqrt{6} u_{0})^2(y) dy]^{\frac{1}{2}}~~
and ~~\Psi(0) = \frac{1}{2}
[\int_{-\infty}^{+\infty}e^{-y}(\sqrt{2}u_{0x}+\sqrt{6} u_{0})^2(y) dy]^{\frac{1}{2}}.$$

The assumption $u_0\neq \lambda e^{-\sqrt{3}x}$ and $u_{0}\in H^s,~~s>\frac{5}{2}$, ensures the validity of estimates (\ref{l64}) with $c_1, c_2 > 0$ for all $t\in [0, T]$.

Now using $\partial_{x}G(x- y) = -\frac{1}{2} sign(x - y)e^{-|x-y|}$ we get
\begin{align}
\nonumber-\int_{0}^{t}
\partial_{x}G \star (\sqrt{2}u_{0x}+\sqrt{6} u_{0})^2(s) ds
&=\frac{e^{-x}t}{2}
\int_{-\infty}^{x}
e^{y}h^2(y, t) dy -
\frac{e^{x}t}{2}
\int_{x}^{+\infty}
e^{-y}h^2(y, t) dy
\\\nonumber&= e^{-x}t
[\Phi(t)-
\frac{1}{2}\int_{x}^{+\infty}
(e^{y} + e^{2x-y})h^2(y, t) dy].
\end{align}

But, by the dominated convergence theorem, it follows
\begin{align}
0 \leq
\int_{x}^{+\infty}
(e^y + e^{(2x-y)})h^2(y, t) dy \leq 2\int_{x}^{+\infty}e^y h^2(y, t) dy\rightarrow 0,~~~ as ~~x \rightarrow +\infty.
\end{align}
In the same way,
$$-\int_{0}^{t}
\partial_{x}G \star (\sqrt{2}u_{0x}+\sqrt{6} u_{0})^2(s) ds = -e^{x}t[\Psi(t) -\frac{1}{2}\int_{-\infty}^{x}(e^{-y} + e^{y-2x})h^2(y, t) dy],
$$
and
$$
0\leq
\int_{x}^{-\infty}(e^{-y} + e^{y-2x})h^2(y, t) dy \leq 2
\int_{x}^{-\infty}2e^{-y}h^2(y, t) dy\rightarrow 0,~~~ as~ ~x \rightarrow-\infty .$$
By Hospital's Rule, we have
\begin{align}
\nonumber\lim_{x\rightarrow\infty}&\frac{\int_{x}^{+\infty}e^{|y|}h^2(t,y)dy}{[\log(1+|x|)]^{1-2d}}\\
=&\frac{-e^{|x|}h^2(t,x)}{(1-2d)[\log(1+|x|)]^{-2d}(1+|x|)^{-1}sign(x)}
\nonumber\\
\leq&C_{2}.
\end{align}

Thus, we complete the asymptotic profile of (\ref{l67}).

\end{proof}

\noindent\textbf{Acknowledgements.} This work was
partially supported by NNSFC (No.11271382), RFDP (No.
20120171110014), the Macao Science and Technology Development Fund (No. 098/2013/A3) and the key project of Sun Yat-sen University.

\phantomsection
\addcontentsline{toc}{section}{\refname}


\begin{thebibliography}{99}
\small

\bibitem{b-g1}S. Baouendi and C. Goulaouic, \textit{Remarks on the abstract form of nonlinear Cauchy¨CKovalevsky theorems}, {Comm. Partial Differential Equations} {\bf2} (1977), 1151--1162.
\bibitem{Brandolese} L.Brandolese, \textit{Break down for the Camassa-Holm equation using decay criteria and persistence in weighted spaces}, {Int.Math.Res.Not.IMRN}, {\bf22}(2012), 5161-5181.


\bibitem{Bressan.Constantin}A. Bressan and A. Constantin, \textit{Global conservative solutions of the Camassa-Holm equation}, {Archive for Rational Mechanics and Analysis,} {\bf183} (2007), 215-239.

\bibitem{Bressan.Constantin2} A. Bressan and A. Constantin, \textit{Global dissipative solutions of the Camassa-Holm equation}, {Analysis and Applications}, {\bf5} (2007), 1-27.

\bibitem{Camassa}R. Camassa and D. D. Holm, \textit{An integrable shallow water equation with peaked solitons}, {Physical Review Letters,} {\bf71}  (1993), 1661-1664.

\bibitem{Camassa.Hyman} R. Camassa, D. Holm and J. Hyman, \textit{A new integrable shallow water
equation}, {Advances in Applied Mechanics}, {\bf31} (1994), 1--33.


\bibitem{C-K}
G. M. Coclite and K. H. Karlsen, On the well-posedness of the
Degasperis-Procesi equation, {\it J. Func. Anal.}, {\bf 233} (2006),
60--91.


\bibitem{Constantin-E} A. Constantin, \textit{The Hamiltonian structure of the Camassa-Holm equation}, {Expositiones Mathematicae}, {\bf 15(1)} (1997), 53-85.

\bibitem{Constantin-P} A. Constantin, \textit{On the scattering problem for the Camassa-Holm equation}, {Proceedings of The Royal Society of London, Series A}, {\bf 457} (2001), 953-970.

\bibitem{Constantin} A. Constantin, \textit{Global existence of solutions and breaking waves for a shallow
water equation: a geometric approach}, Annales de l'Institut Fourier (Grenoble),
{\bf50} (2000), 321-362.

\bibitem{Constantin2} A. Constantin, \textit{The trajectories of particles in Stokes waves}, {Inventiones Mathematicae}, {\bf166} (2006), 523-535.

\bibitem{Constantin1}
A. Constantin,  \textit{Finite propagation speed for the Camassa-Holm equation}, {J.
Math. Phys.}, {\bf 46} (2005), 023506.


\bibitem{Constantin.Escher} A. Constantin and J. Escher, \textit{Global existence and blow-up for a shallow water equation}, {Annali della Scuola Normale Superiore di Pisa - Classe di Scienze},  {\bf26} (1998), 303-328.

\bibitem{Constantin.Escher2} A. Constantin and J. Escher, \textit{Well-posedness, global existence, and
blowup phenomena for a periodic quasi-linear hyperbolic equation}, {Communications on Pure and Applied Mathematics}, {\bf51} (1998), 475-504.

\bibitem{Constantin.Escher3} A. Constantin and J. Escher, \textit{Wave breaking for nonlinear nonlocal shallow water equations}, {Acta Mathematica}, {\bf181} (1998), 229-243.

\bibitem{Constantin.Escher4} A. Constantin and J. Escher, \textit{Particle trajectories in solitary water waves}, {Bulletin of the American Mathematical Society}, {\bf44} (2007), 423-431.

\bibitem{Constantin.Escher5} A. Constantin and J. Escher, \textit{Analyticity of periodic traveling free surface water waves with vorticity}, {Annals of Mathematics}, {\bf173} (2011), 559-568.


\bibitem{Constantin.Lannes} A. Constantin and D. Lannes, \textit{The hydrodynamical relevance of the
Camassa-Holm and Degasperis-Procesi equations}, {Archive for Rational Mechanics and Analysis}, {\bf192} (2009), 165-186.

\bibitem{Constantin.mckean} A. Constantin and H. P. McKean, \textit{A shallow water equation on the circle}, {Comm. Pure Appl. Math.}, {\bf55} (1999), 949--982.


\bibitem{Constantin.Molinet} A. Constantin and L. Molinet, \textit{Global weak solutions for a shallow water equation}, {Communications in Mathematical Physics}, {\bf211} (2000), 45-61.

\bibitem{Constantin.lvanov.lenells} A. Constantin, R. I. Ivanov and J. Lenells, \textit{Inverse scatterering transform for the Degasperis-Procesi equation}, {Nonlinearity}, {\bf23} (2010), 2559--2575.

\bibitem{Constantin.Strauss}  A. Constantin and W. A. Strauss, \textit{Stability of peakons}, {Communications on Pure and Applied Mathematics}, {\bf53} (2000), 603-610.




\bibitem{d1}R. Danchin, \textit{A few remarks on the Camassa-Holm equation}, {Differential Integral Equations}, {\bf14} (2001), 953-988.

\bibitem{D-H-H}
A. Degasperis, D. D. Holm, and A. N. W. Hone, {\it A new integral
equation with peakon solutions}, {Theor. Math. Phys.}, {\bf 133}
(2002), 1463--1474.

\bibitem{D-P}
A. Degasperis and M. Procesi, {\it Asymptotic integrability}, { Symmetry
and Perturbation Theory}, {\bf1(1)}(1999), 23--37.

\bibitem{D-G-H}
H. R. Dullin, G. A. Gottwald, and D. D. Holm, {\it On asymptotically
equivalent shallow water wave equations}, {Phys. D}, {\bf 190}
(2004), 1--14.

\bibitem{E-L-Y1}
J. Escher, Y. Liu and Z. Yin, {\it Global weak solutions and blow-up structure for the Degasperis-Procesi  equation}, {J.
Funct. Anal.}, {\bf 241} (2006), 457--485.

\bibitem{E-L-Y}
J. Escher, Y. Liu and Z. Yin, {\it Shock waves and blow-up phenomena for the periodic Degasperis-Procesi  equation}, {Indiana Univ. Math. J.}, {\bf 56} (2007), 87--177.

\bibitem{Fokas} A. Fokas and B. Fuchssteiner, \textit{Symplectic structures, their B\"{a}cklund transformation and hereditary symmetries}, {Physica D}, {\bf 4(1)} (1981/82), 47--66.

\bibitem{G-L} G. Gui and Y. Liu, \textit{On the Cauchy problem for the Degasperis-Procesi equation}, {Quart. Appl. Math.}, {\bf 69} (2011), 445-464.



\bibitem{H-H} A. A. Himonas and C. Holliman, \textit{The Cauchy problem for the Novikov equation}, {Nonlinearity}, {\bf 25} (2012), 449-479.
\bibitem{h-m1} A.A. Himonas and G. Misiolek, \textit{Analyticity of the Cauchy problem for an integrable evolution equation}, {Math. Ann.}, {\bf327} (2003), 575--584.
\bibitem{H-M-P}
A. Himonas, G. Misiolek, G. Ponce and Y. Zhou, \textit{Persistence properties and unique continuation of
solutions of the Camassa-Holm equation}, { Comm. Math. Phys.}, {\bf 27} (2007), 511-522.


\bibitem{Hone}A. N. W. Hone and J. Wang, \textit{Integrable peakon equations with cubic nonlinearity}, {Journal of Physics A: Mathematical and Theoretical}, {\bf41} (2008), 372002, 10pp.

\bibitem{Laishaoyong}S. Lai, \textit{Global weak solutions to the Novikov equation}, {Journal of Functional Analysis}, {\bf265} (2013), 520--544.


\bibitem{Le}
J. Lenells, {\it Traveling wave solutions of the Degasperis-Procesi
equation}, {J. Math. Anal. Appl.}, {\bf 306} (2005), 72--82.

\bibitem{L-Y1}
Y. Liu and Z. Yin, {\it Global Existence and Blow-up Phenomena for the
Degasperis-Procesi Equation}, {Commun. Math. Phys.}, {\bf 267} (2006), 801--820.

\bibitem{L-Y2} Y. Liu and Z. Yin,  \textit{On the blow-up phenomena for the Degasperis-Procesi equation}, {Int. Math. Res. Not. IMRN}, {\bf 23} (2007), rnm117, 22 pp.

\bibitem{Lu}
H. Lundmark, {\it Formation and dynamics of shock waves in the
Degasperis-Procesi equation}, {J. Nonlinear. Sci.}, {\bf 17}
(2007), 169--198.

\bibitem{Lu-S}
H. Lundmark and J. Szmigielski, {\it Multi-peakon solutions of the
Degasperis-Procesi equation}, {Inverse Prob.}, {\bf 19} (2003), 1241--1245.



\bibitem{n1}V. Novikov, \textit{Generalization of the Camassa-Holm equation}, {J. Phys. A}, {\bf 42} (2009), 342002, 14 pp.

\bibitem{Guillermo} G. Rodr\'{i}guez-Blanco, \textit{On the Cauchy problem for the Camassa-Holm equation},
{Nonlinear Analysis: Theory Methods Application}, {\bf46} (2001), 309-327.



 \bibitem{Toland}  J. F. Toland,  \textit{Stokes waves}, Topological Methods in Nonlinear Analysis, {\bf7} (1996), 1-48.

\bibitem{V-P}
V. O. Vakhnenko and E. J. Parkes,  \textit{Periodic and solitary-wave solutions of the Degasperis-Procesi equation}, {Chaos Solitons
Fractals}, {\bf 20} (2004), 1059--1073.

\bibitem{Wu.Yin} X. Wu and Z. Yin, \textit{Global weak solutions for the Novikov equation}, {Journal of Physics A: Mathematical and Theoretical}, {\bf44} (2011), 055202, 17pp.

\bibitem{Wu.Yin2}X. Wu and Z. Yin, \textit{Well-posedness and global existence for the Novikov equation}, {Annali della Scuola Normale Superiore di Pisa. Classe di Scienze. Serie V}, {\bf11} (2012), 707-727.

\bibitem{Wu.Yin3}X. Wu and Z. Yin, \textit{A note on the Cauchy problem of the Novikov equation}, {Applicable Analysis}, {\bf92} (2013), 1116--1137.

\bibitem{Xin.Z.P}Z. Xin and P. Zhang, \textit{On the weak solutions to a shallow water equation}, {Communications on Pure and Applied Mathematics}, {\bf53} (2000), 1411--1433.

\bibitem{y-y1} K. Yan and Z. Yin, \textit{Analytic solutions of the Cauchy problem for two-component shallow water systems}, {Math. Z.}, (2010), doi:10.1007/s00209-010-0775-5.


\bibitem{Wei.Yan}W. Yan, Y. Li and Y. Zhang, \textit{The Cauchy problem for the integrable Novikov equation}, {Journal of Differential Equations}, {\bf253} (2012), 298--318.

\bibitem{Wei.Yan2}W. Yan, Y. Li and Y. Zhang, \textit{The Cauchy problem for the Novikov equation}, {Nonlinear Differential Equations and Applications NoDEA}, {\bf20} (2013), 1157--1169.

\bibitem{y1}Z.  Yin, \textit{On the Cauchy problem for an integrable equationwith peakon solutions}, {Ill. J. Math.}, {\bf 47} (2003), 649--666.

\bibitem{y2} Z. Yin, \textit{ Global existence for a new periodic integrable equation}, {J. Math. Anal. Appl.}, {\bf 283} (2003), 129--139.

\bibitem{y3} Z. Yin, \textit{ Global weak solutions to a new periodic integrable equation with peakon solutions}, {J. Funct. Anal.}, {\bf 212} (2004), 182--194.

\bibitem{y4} Z. Yin, \textit{ Global solutions to a new integrable equation with peakons}, {Indiana Univ. Math. J.}, {\bf 53} (2004), 1189--1210.

\bibitem{Tu-Yin1} X. Tu and Z. Yin,  \textit{Blow-up phenomena and local well-posedness for a generalized Camassa-Holm equation with peakon solutions}, to appear in {Discrete and Continuous Dynamical Systems - Series A}.

\bibitem{Tu-Yin2} X. Tu and Z. Yin,  \textit{Blow-up phenomena and local well-posedness for a generalized Camassa-Holm equation in the critical Besov space}, to appear in {Nonlinear Analysis Series A: Theory, Methods $\&$ Applications}.






\end{thebibliography}
\end{document}